\documentclass[a4paper,10pt]{article}
\usepackage{amssymb,amsmath}
\usepackage{amsfonts}
\usepackage[english]{babel}
\usepackage[utf8]{inputenc}
\usepackage{fancyhdr}
\usepackage{yfonts}
\usepackage{mathrsfs}
\usepackage{graphicx}
\usepackage{enumerate}
\usepackage{xspace}
\usepackage{cite}
\usepackage{psfrag}
\usepackage{url}
\usepackage{enumitem}

\usepackage[dvipsnames]{xcolor}
\usepackage{xcolor}
\usepackage{lipsum}

\usepackage{verbatim}
\usepackage{amsthm}

\providecommand{\com[2]}{\begin{tt}[#1: #2]\end{tt}}
\usepackage{color}

\newtheorem{theorem}{Theorem}
\newtheorem{lemma}[theorem]{Lemma}
\newtheorem{definition}[theorem]{Definition}
\newtheorem{proposition}[theorem]{Proposition}

\providecommand{\prt}[1]{\left( #1 \right)}
 
\providecommand{\PP}{\mathbb{P}}

\providecommand{\abs[1]}{\left|#1\right|}
\providecommand{\norm[1]}{\abs{\abs{#1}}}%
\providecommand{\floor[1]}{\lfloor  #1 \rfloor }

\providecommand{\1}{\textbf{1}}
\providecommand{\0}{\textbf{0}}\providecommand{\quotes[1]}{``#1"}
\def\0{{\bf 0}}
\def\R{\mathbb{R}}

\def\Z{\mathbb{Z}}

\def\E{\mathbb{E}}

\def\cA{\mathcal{A}}

\providecommand{\forbid}{R_\gamma}
\providecommand{\smallangles}{B_\gamma}
\providecommand{\largeangles}{U_\gamma}
\providecommand{\Dm}{D(\mu)} 
\renewcommand{\Re}{{\mathrm Re}}

\providecommand{\cyclemod}{reconnected cycle }
\providecommand{\arcmod}{reconnected arcs }
\providecommand{\Cyclemod}{Reconnected Cycle }
\providecommand{\Arcmod}{Reconnected Arcs }

\begin{document}

\title{Improved Mixing Rates of Directed Cycles\\ with Additional Sparse Interconnections}

\author{Balázs Gerencsér\thanks{B. Gerencs\'er is with the HUN-REN Alfr\'ed R\'enyi Institute of Mathematics, Budapest, Hungary and E\"otv\"os Lor\'and University, Department of Probability and Statistics, Budapest, Hungary (email: {\tt\small gerencser.balazs@renyi.hu}). His research was supported by the J\'anos Bolyai Research Scholarship of the Hungarian Academy of Sciences.}
\and Julien M. Hendrickx\thanks{J. Hendrickx is with the ICTEAM institute, UCLouvain, Belgium. (email: {\tt\small julien.hendrickx@uclouvain.be}). His research is supported by the Incentive Grant
for Scientific Research (MIS) ``Learning from Pairwise Data'', by the KORNET project from F.R.S.-FNRS and by the ``RevealFlight'' ARC at UCLouvain.}}

\maketitle

\begin{abstract}
We analyze the absolute spectral gap of Markov chains on graphs obtained from a cycle of $n$ vertices and perturbed only at approximately $n^{1/\rho}$ random locations with an appropriate, possibly sparse, interconnection structure. Together with a strong asymmetry along the cycle, the gap of the resulting chain can be bounded inversely proportionally by the longest arc length (up to logarithmic factors) with high probability, providing a significant mixing speedup compared to the reversible version.
\end{abstract}

%%%%%%%%%%%%%%%%%%%%%%%%%%%%%%%%%%%%%%%%%%%%%%%%%%%%%%%%%%%%%
%%%%%%%%%%%%%%%%%%%%%%%%%%%%%%%%%%%%%%%%%%%%%%%%%%%%%%%%%%%%%
%%%%%%%%%%%%%%%%%%%%%%%%%%%%%%%%%%%%%%%%%%%%%%%%%%%%%%%%%%%%%
\section{Introduction}
%%%%%%%%%%%%%%%%%%%%%%%%%%%%%%%%%%%%%%%%%%%%%%%%%%%%%%%%%%%%%
%%%%%%%%%%%%%%%%%%%%%%%%%%%%%%%%%%%%%%%%%%%%%%%%%%%%%%%%%%%%%

This work focuses on understanding how the mixing properties of Markov chains, and how they can be improved by making as little change as possible on the initial connection graph. 
Some impressive results in that direction are already available. The authors of \cite{krivelevich2015smoothed} show how a rather general class of initial graphs on $n$ vertices can be turned into an expander, thereby making the simple random walk rapidly mixing, by adding approximately $\epsilon n $ random edges. A more refined analysis resulting in stronger bounds is performed in \cite{hermon2022universality} and its follow-up \cite{baran2025phase}, where a (possibly weighted) random perfect matching is added. These authors are then able to determine a mixing-time with cut-off. However, this line of works always involves adding a number of edges that \emph{grows linearly} with the number of vertices. By contrast, we aim at understanding if significant speed-ups can still be obtained by adding a much smaller number edges that create shortcuts, provided \emph{we adapt the transition probabilities} of the initial chain, specifically without requiring them to remain reversible.

Since determining the mixing time is currently out of reach in the settings we consider, we will aim at approximating the (absolute) spectral gap, 
    measuring the mixing efficiency of the Markov chain in the asymptotic regime.
\begin{definition}
  For any irreducible doubly stochastic matrix $A$ let its \emph{spectral gap} be
    \begin{equation}\label{def:spectral_gap}
    \lambda(A) = \min\left\{1-|\mu|~:~\mu\neq 1 \text{ is an eigenvalue of }
    A \right\}.\end{equation}
\end{definition}

We focus on the cycle graph, and show that the combination of transition probabilities adaptation and addition of a small number of edges does indeed significantly outperform any of these two operations taken individually. Indeed, it was shown in \cite{gerencser2011markov} that optimizing the transition probabilities on a cycle graph could not provide any significant speedup even when allowing for non-reversibility. On the other hand, one can verify that in the reversible case, the spectral gap of a cycle on $n$ nodes to which one adds edges incident to $k\ll n$ nodes will always remain $O(k^2/n^2)$.
By contrast, in our recent work \cite{gerencser2019improved}, we have shown that a significant speed-up to $\Theta(k\log^{-\gamma} k/n)$ (for $\gamma>4$) could be obtained if a small number $k\approx n^{1/\rho}$ (for $\rho>1$) of randomly selected positions were perturbed and interconnected by a \emph{fully connected graph}, and the walk on the rest of the cycle was made directed and deterministic.
The current paper generalizes these results by showing that the speed-up is present for more general interconnection structures, provided it is sufficiently expanding, see Definition \ref{def:cyclemod} and Theorem \ref{thm:main} for a precise statement of the model and the result.

We note the existence of several alternative ways of improving the spectral gaps of Markov chains. The first one relies on optimizing the individual transition probabilities in an algorithmically accessible way \cite{boyd_and_al:fastmix2004,boyd_and_al:fastmix2009}, which then can be then exploited and carried further for special cases, e.g.\ for random geometric graphs \cite{boydrng2005}. But these works almost all impose the reversibility of the Markov chain. Moreover, we have seen above that in the case of the cycle graph, optimization of the probabilities alone cannot lead to any significant speed-up, even in the absence of reversibility requirements. 
In fact, it can even be shown that appropriate homogeneous reversible random walks are optimal in some sense \cite{pauljulien2015spectralgap}.
A different recent technique is based on alternating between stepping with the original Markov chain and moving by a permutation fixed apriori, see \cite{chatterjee2020speeding, ben2023cutoff, he2022markovpermutation} for the initial theory, refined cut-off analysis for random permutations, and some specialized case study.  But this does clearly involve very significant modification of the Markov chain structure.  
Another possibility is to reduce the backtracking of the chain, in \cite{diaconis2013spectral} a non-backtracking component is introduced and indeed an improvement of the spectral gap is shown.

Our paper is organized as follows: We introduce our precise models of Markov chains and state our main results in Section \ref{sec:Model-Results}. We then show in Section \ref{sec:lowdimensional} how the spectral gap can be characterized in terms of a lower-dimensional modified eigenvalue problem, and what it entails for the corresponding eigenvectors. Sections \ref{sec:small_angles} and \ref{sec:large_angles} focus then on the analysis of eigenvalues with respectively small and large arguments, and their combinations allows establishing the results of Section \ref{sec:Model-Results}. We compare our results with experimental studies in Section \ref{sec:numerics} and observe that our bounds (without the polylogarithmic factors) appear to predict the actual behavior of the spectral gap. We finish by discussing our result and some possible further works in Section \ref{sec:ccl}.

%%%%%%%%%%%%%%%%%%%%%%%%%%%%%%%%%%%%%%%%%%%%%%%%%%%%%%%%%%%%%
%%%%%%%%%%%%%%%%%%%%%%%%%%%%%%%%%%%%%%%%%%%%%%%%%%%%%%%%%%%%%
%%%%%%%%%%%%%%%%%%%%%%%%%%%%%%%%%%%%%%%%%%%%%%%%%%%%%%%%%%%%%
\section{Models and Main Results}\label{sec:Model-Results}
%%%%%%%%%%%%%%%%%%%%%%%%%%%%%%%%%%%%%%%%%%%%%%%%%%%%%%%%%%%%%
%%%%%%%%%%%%%%%%%%%%%%%%%%%%%%%%%%%%%%%%%%%%%%%%%%%%%%%%%%%%%

%%%%%%%%%%%%%%%%%%%%%%%%%%%%%%%%%%%%%%%%%%%%%%%%%%%%%%%%%%%%%
%%%%%%%%%%%%%%%%%%%%%%%%%%%%%%%%%%%%%%%%%%%%%%%%%%%%%%%%%%%%%
\subsection{Models}
%%%%%%%%%%%%%%%%%%%%%%%%%%%%%%%%%%%%%%%%%%%%%%%%%%%%%%%%%%%%%

We now formally define the random Markov transition matrix model outlined in the introduction.

\begin{definition}[\cyclemod model]\label{def:cyclemod}
Let $n\ge k$ and $A\in \R^{k\times k}$ be a symmetric doubly stochastic matrix. We define the \emph{\cyclemod} $H_n(A)$ as the random Markov transition matrix corresponding to the directed weighted graph obtained by the three steps below, depicted in Figure \ref{fig:3steps_H}
\begin{enumerate}[wide, labelwidth=!, labelindent=4pt,topsep=3pt,itemsep=0pt]
    \item Start with a directed cycle on $n$ vertices where every edge has a weight $1$.
    \item Uniformly randomly select $k$ distinct edges. Randomly pick one of them to be denoted by $(e_1,b_2)$ and label the subsequent ones (following the order of the cycle) by $(e_i,b_{i+1})$, $i=1,\dots, k$, with $k+1\equiv 1$. Then, remove these edges. 
    \item For each $i,j\in 1,\dots,k$, add a directed edge from $e_j$ to $b_{i+1}$ with weight $A_{ij}=A_{ji}$ (this only creates an actual edge when $A_{ij}\neq 0$)
\end{enumerate}
\end{definition}

\begin{figure}
    \centering
    \includegraphics[width=0.75\textwidth]{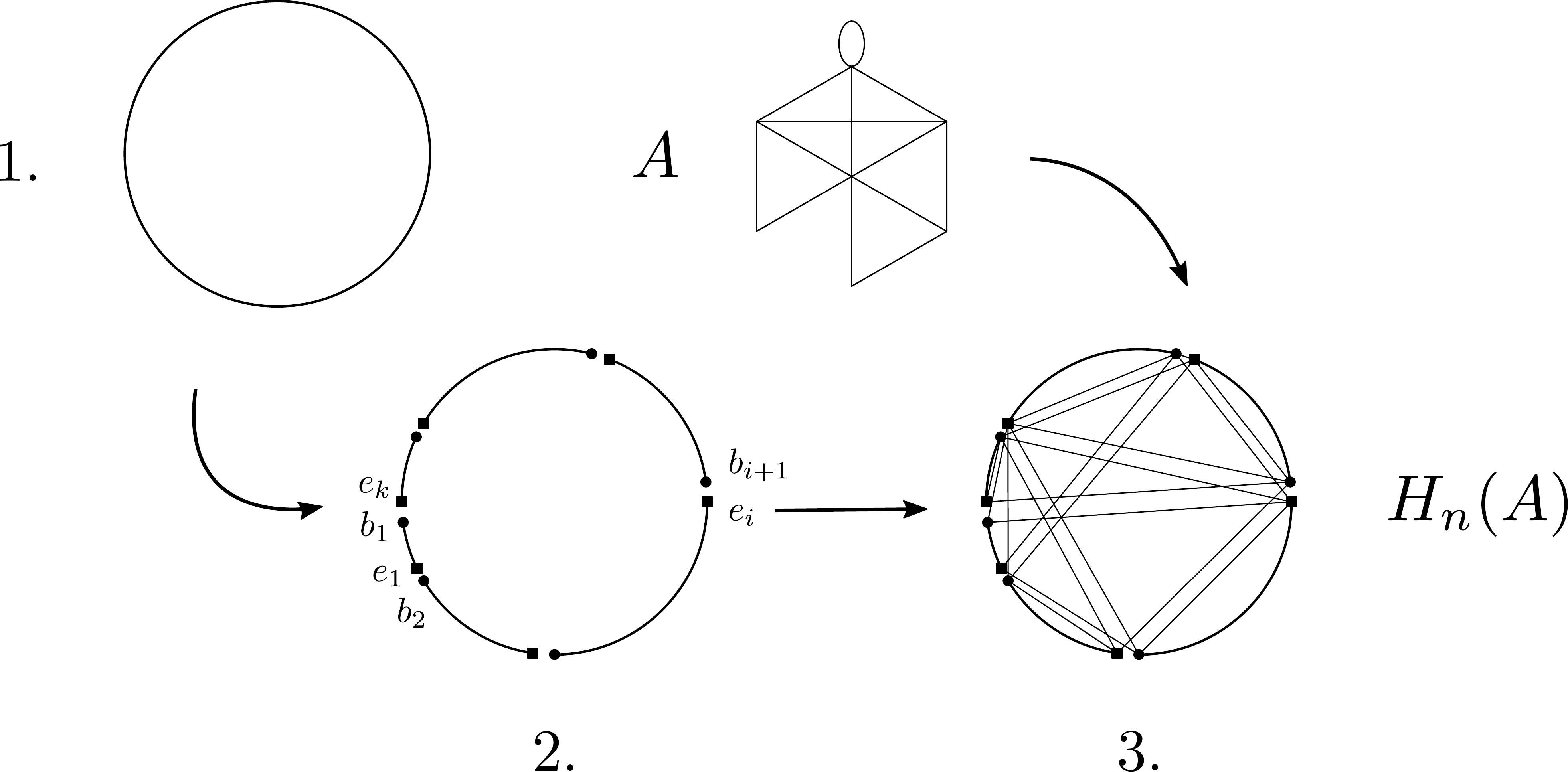}
    \caption{Representation of the random process generating the \cyclemod $H_n(A)$ using $A$ according to Definition \ref{def:cyclemod}.}
    \label{fig:3steps_H}
\end{figure}

Since $A$ is doubly stochastic, one directly verifies that $H_n(A)$ is also doubly stochastic. However, it is (highly) non-symmetric. 
Observe that the length of the arcs between the removed edges $(e_i,b_{i+1})$ are not independent because their sum is fixed, which creates difficulties in the analysis. To circumvent this problem, we will first study a related models with independent arc lengths, but for which the total size is no longer pre-determined. 

\providecommand{\Lav}{\bar L}
\providecommand{\Lsqav}{\overline {L^2}}

\begin{definition}[\arcmod model]\label{def:arcmod}
Let $L>0$ and $A\in \R^{k\times k}$ be a symmetric doubly stochastic matrix for some integer $k$. We define the \emph{\arcmod} $G_L(A)$ as the random Markov transition matrix corresponding to the directed weighted graph obtained by the three steps below. 
\begin{enumerate}[wide, labelwidth=!, labelindent=4pt,topsep=3pt,itemsep=0pt]
\item Choose $k$ independent random values $L_i$ from the $Geo(1/L)$ distribution (taking minimal value 1), with $i =1,\dots k$.
\item For each $i=1,\dots,k$, build a directed arc of $L_i$ nodes (and thus $L_i-1$ edges) connected by edges of weight 1. We use a double index for the nodes: $(i,1), (i,2), \dots (i,L_i)$, and there is thus a transition probability 1 from each $(i,\ell)$ to $(i,\ell+1)$ when $\ell< L_i$.
\item For each $i,j\in 1,\dots,k$, add a directed edge from $(j,L_j)$ to $(i+1,1)$ with weight $A_{ij}=A_{ji}$ (this only creates an actual edge when $A_{ij}\neq 0$)
\end{enumerate}
\end{definition}
Again, one can directly verify that $G_L(A)$ is doubly stochastic. Moreover, there is a parallel between the two models, with the nodes $(j,L_j)$ and $(i+1,1)$ in $G_L(A)$ playing a role similar as $e_j$ and $b_{i+1}$ in $H_n(A)$. The exact correspondence will be detailed in Lemma \ref{lem:modelconditioning}. In the sequel we will denote by $(L_i)$ the collection of the $k$ random variables $L_1,\dots,L_k$.

Our results will be asymptotic in $k$ and $n$ or $L$. Hence we need to define a class allowable matrices $A$

\begin{definition}\label{def:class_A}
    We call $\cA_{c,\alpha}$ the set of symmetric doubly-stochastic matrices $A$ for which $\lambda(A)\geq \Delta_k:= c \log^{-\alpha} k$ for some fixed $c>0,\alpha \geq 0$ independent of $A$ and $k$, where $k$ is the dimension of the matrix and $\lambda(.)$ was defined in  \eqref{def:spectral_gap}.
\end{definition}

%%%%%%%%%%%%%%%%%%%%%%%%%%%%%%%%%%%%%%%%%%%%%%%%%%%%%%%%%%%%%
%%%%%%%%%%%%%%%%%%%%%%%%%%%%%%%%%%%%%%%%%%%%%%%%%%%%%%%%%%%%%
\subsection{Main Results}
%%%%%%%%%%%%%%%%%%%%%%%%%%%%%%%%%%%%%%%%%%%%%%%%%%%%%%%%%%%%%

We first state our theorem for our proxy \arcmod model $G_L(A)$ with independent arc lengths, which also be considered first in our analysis.
\begin{theorem}
  \label{thm:indep}
  Assume $k,L \rightarrow \infty$ in parallel while $\rho_l < \log L / \log k < \rho_u$ for some constants $0<\rho_l<\rho_u<\infty$, and a matrix series $A_k\in\cA_{c,\alpha}$ is given for some $c,\alpha$. For any constant $\gamma > 7 + \alpha$ the spectral gap of the random transition probability matrix $G_L(A_k)$ can be asymptotically almost surely (a.a.s.) as follows:
  $$
  \lambda(G_L(A_k)) \ge \frac{1}{L\log^\gamma k}.
  $$
\end{theorem}

The partial results we develop for proving Theorem \ref{thm:indep} for the \arcmod model will allow deducing, with some additional analysis, an analogous result for the \cyclemod model.

\begin{theorem}
  \label{thm:main}
  Assume $n,k \rightarrow \infty$ in parallel while $\rho_l < \log n / \log k < \rho_u$ for some constants $1<\rho_l<\rho_u<\infty$, and a matrix series $A_k\in\cA_{c,\alpha}$ is given for some $c,\alpha$. For any constant $\gamma > 7 + \alpha$ the spectral gap of the random transition probability matrix $H_n(A_k)$ can be a.a.s.\ bounded as follows:
  $$
  \lambda(\bar A) \ge \frac{k}{n\log^\gamma k}.
  $$
\end{theorem}

These rates are comparable, up to logarithmic factors, to those obtained in \cite{gerencser2019improved}, which was restricted to the fully connected $A=\frac{1}{k}\1\1^T$. It represents a significant improvement with respect 
to the symmetrized version, where it is easy to see that the spectral gap is $O(k^2/n^2)$, showing the very positive impact that asymmetry can have on mixing properties.
The next three sections are devoted to proving these two theorems.

%%%%%%%%%%%%%%%%%%%%%%%%%%%%%%%%%%%%%%%%%%%%%%%%%%%%%%%%%%%%% 
%%%%%%%%%%%%%%%%%%%%%%%%%%%%%%%%%%%%%%%%%%%%%%%%%%%%%%%%%%%%% 
%%%%%%%%%%%%%%%%%%%%%%%%%%%%%%%%%%%%%%%%%%%%%%%%%%%%%%%%%%%%% 
\section{Low-Dimensional Representation for \Arcmod}\label{sec:lowdimensional}
%%%%%%%%%%%%%%%%%%%%%%%%%%%%%%%%%%%%%%%%%%%%%%%%%%%%%%%%%%%%% 
%%%%%%%%%%%%%%%%%%%%%%%%%%%%%%%%%%%%%%%%%%%%%%%%%%%%%%%%%%%%% 

%%%%%%%%%%%%%%%%%%%%%%%%%%%%%%%%%%%%%%%%%%%%%%%%%%%%%%%%%%%%% 
%%%%%%%%%%%%%%%%%%%%%%%%%%%%%%%%%%%%%%%%%%%%%%%%%%%%%%%%%%%%%
%%%%%%%%%%%%%%%%%%%%%%%%%%%%%%%%%%%%%%%%%%%%%%%%%%%%%%%%%%%%%
\subsection{Modified Eigenvalue Problem}
%%%%%%%%%%%%%%%%%%%%%%%%%%%%%%%%%%%%%%%%%%%%%%%%%%%%%%%%%%%%%
%%%%%%%%%%%%%%%%%%%%%%%%%%%%%%%%%%%%%%%%%%%%%%%%%%%%%%%%%%%%%

On the way to find the eigenvalues - eigenvectors of the \arcmod $G_L(A)$ described in Definition \ref{def:arcmod}, our first step will be to transform $G_L(A) y = \mu y$ into a condensed equivalent form involving only $A$ and one variable per arc. 

\begin{proposition}
The eigenvalues $\mu$ of \arcmod $G_L(A)$ are exactly the values for which the following equation admits a non-trivial solution $x$
    \begin{equation}\label{eq:eig_main}
 CAx = \Dm x,
\end{equation}
where $\Dm$ is the diagonal matrix with $D_{ii} = \mu^{L_i}$ and $C$ is a circular permutation matrix with $C_{i,i-1}=1$ for all $i$ (with $k+1\equiv 1$).
\end{proposition}
\begin{proof}
    Consider an eigenvalue $\mu$ and corresponding eigenvector $y$ for $G_L(A)$ thus satisfying $G_L(A) y = \mu y$. Starting from a node $(i,\ell)$ for $\ell >1$, since this node only has one incoming edge from $(i,\ell-1)$, the equation $G_L(A) y = \mu y$ implies $\mu y_{i,\ell} = y_{i,\ell-1}$. The value of all the $y_{i,\ell}$ is thus uniquely specified by $y_{i,L_i}$. Defining $x_i:=y_{i,L_i}$, we indeed have $y_{i,\ell} = x_i \mu^{L_i-\ell}$. In particular, $y_{i,1} = x_i \mu^{L_i-1}$.

    Let us now consider node $y_{i,1}$. This node has incoming edges with weights $A_{i-1,j}$ from all nodes $(j,L_j)$ for which $A_{i-1,j}\neq 0$. The eigenvalue equation yields thus 
$$\mu y_{i,1}   = \sum_j A_{i-1,j} y_{j,L_j}.$$
Taking into account the  relation $y_{i,1} = x_i \mu^{L_i-1}$, we have equivalently 
$$
\mu \cdot \mu^{L_i-1} x_i = \sum _j A_{i-1,j} x_j,  \hspace{.5cm} \forall i=1,\dots,k,
$$
which can be rewritten as
\begin{equation}\label{eq:eig_main}
 CAx = \Dm x.
\end{equation}
So any solution of $G_L(A)  y = \mu y$ yields a solution of \eqref{eq:eig_main}, and by expanding $x$ to $y$ it is easy to confirm that the converse is also true. 
\end{proof}

The sequel of the proof focuses thus on this condensed equation \eqref{eq:eig_main}. We will show the impossibility, w.h.p. of obtaining solution $x,\mu$ ($\mu\neq 1$) of \eqref{eq:eig_main} with $\mu\in \forbid$, defined as 
$$
\forbid := \{\mu\neq 1\mid 1\geq |\mu| \geq 1-\frac{1}{L\log^\gamma k}\}.
$$
We will in fact follow different approaches for $\mu$ with \quotes{small} or \quotes{large} arguments, and consider thus respectively in Section \ref{sec:small_angles} and Section \ref{sec:large_angles} the sets
\begin{align}
\smallangles &:= \{\mu\in \forbid, |\arg \mu|\leq  \frac{\pi}{M_\varphi L \log k}\},\\
\largeangles &:= \{\mu\in \forbid, |\arg \mu|>  \frac{\pi}{M_\varphi L \log k}\},
\end{align}
for a constant $M_\varphi$ to be chosen later, with obviously $\forbid = \smallangles \cup \largeangles$.
Our analysis will be valid under the assumption that the arcs are not exceptionally long which is characterized by the event
\providecommand{\BL}{S(M)} 
\begin{equation}\label{eq:def_BL}
\BL:= \{\max_i L_i \leq ML\log k\},
\end{equation}
for a constant $M>1$ to be chosen later.

Before going further, we need certain technical lemmas establishing that $\BL$ holds with high probability and deriving approximations related to $\mu^{L_i}$ that are valid in that case.

%%%%%%%%%%%%%%%%%%%%%%%%%%%%%%%%%%%%%%%%%%%%%
%%%%%%%%%%%%%%%%%%%%%%%%%%%%%%%%%%%%%%%%%%%%%
\subsection{High Probability Approximations}
%%%%%%%%%%%%%%%%%%%%%%%%%%%%%%%%%%%%%%%%%%%%%

We first re-state a lemma from \cite{gerencser2019improved}, characterizing the probability of having no exceptionally long arc length in the \arcmod model.
\begin{lemma}\label{lem:bound_Li}
Consider the event $\BL$ defined in \eqref{eq:def_BL} with $M>1$.
This event $\BL$ occurs asymptotically almost surely, as its probability satisfies
$$
P(\BL) \geq 1 - 2k^{1-M}.
$$
\end{lemma}
\begin{proof}
  Assume $k\ge 2$. Remember that $(L_i)$ are i.i.d.\ variables with law
  $Geo(1/L)$, thus for any of them we have that
  $$P(L_i \ge ML \log k) = P(L_i \ge \lceil ML \log k \rceil)
  = \left(1-\frac{1}{L}\right)^{\lceil ML \log k\rceil-1}
  \le 2\left(1-\frac{1}{L}\right)^{ML \log k},$$  
  based on $\left(1-\frac{1}{L}\right)^{\lceil ML \log k\rceil}
  \le \left(1-\frac{1}{L}\right)^{ML \log k}$ and
  $\left(1-\frac{1}{L}\right)^{-1}<2$.
  Knowing $(1-1/L)^L < 1/e$ we get
  $$P(L_i \ge ML \log k) \le 2e^{-M\log k} = 2k^{-M}.$$
  To treat all $(L_i)$ together for $1\le i \le k$, we use a simple union bound.
  $$P(\max_i L_i \ge ML \log k) \le 2k^{1-M}.$$
\end{proof}

We now provide some linear approximations of quantities involving $\mu^{L_i}$ that are valid in the set $\forbid$ (and specifically $\smallangles$), where we aim to establish the absence of eigenvalues.

\begin{lemma}
  \label{lem:muL_abs}
If $\mu \in \forbid$, then for $k$ large enough under event $\BL$ defined in \eqref{eq:def_BL} there holds
\begin{equation}
  \label{eq:abssq_bound}
  \abs{\mu^{L_i}}^2\geq 1-4M\log^{-(\gamma-1)} k =: 1-\epsilon_k.
\end{equation}
\end{lemma}
\begin{proof}
  By the definition of $\forbid$ we have
  $$
  \abs{\mu^{L_i}}^2 \geq \abs{1-\frac{1}{L\log^\gamma k}}^{2ML\log k} > \exp(-4M\log^{-(\gamma-1)} k) > 1-4M\log^{-(\gamma-1)} k.
  $$
  Here we simply use the linear approximation of the exponential function near 0, in particular $k$ has to be large enough for the middle inequality.
\end{proof}

\begin{lemma}\label{lem:linear_small_angles}
If $\mu \in \smallangles$, with $M_\varphi \geq 3\pi M$ then for $k$ large enough under event $\BL$ there holds
\begin{equation}
  \label{eq:linear_small_angles}
  \left|(1-\mu^{L_i}) - L_i(1-\mu) \right| \leq \frac 1 2 L_i|1-\mu|.  
\end{equation}
Furthermore this leads to
\begin{align}
  \left|\sum_i(1-\mu^{L_i})\right| &\ge \frac 1 2 |1-\mu| \left(\sum_i L_i\right), \label{eq:linear_sum}\\
  \sum_i\left|1-\mu^{L_i}\right|^2 &\le 3 |1-\mu|^2 \left(\sum_i L_i^2\right).\label{eq:linear_squares}
\end{align}
\end{lemma}
\begin{proof}
For the main inequality we rearrange the left hand side to write
\begin{equation}\label{eq:first_eq_lemma_linmu}
|(1-\mu^{L_i}) - L_i(1-\mu) | = \left|\frac{1-\mu^{L_i}}{1-\mu} - L_i\right| \cdot |1-\mu| = \left|\sum_{j=0}^{L_i-1}\mu^j - L_i\right| \cdot |1-\mu|.
\end{equation}
Using the triangular inequality and the fact that $|1-e^{i\theta}|\leq |\theta|$, we obtain 
$$
|\mu^j-1|\leq |\mu^j - e^{i \arg \mu^j}| + |e^{i \arg \mu^j}-1|\leq 1 - \abs{\mu^j} + \abs{\arg \mu ^j}.
$$
Moreover, it follows from the definition of $\smallangles$ that
\begin{equation}\label{eq:bound_arg}
|\arg \mu^j| \leq j \abs{\arg \mu} \leq j\frac{\pi}{M_{\varphi}L\log k} \leq \frac{M\pi}{M_\varphi}\leq \frac{1}{3},
\end{equation}
where we have used $M_\varphi \geq 3\pi M$ and $j\leq L_i\leq ML \log k$ under $\BL$.
On the other hand, using Lemma \ref{lem:muL_abs}
$$
1 - |\mu^j|\leq 1-{\mu^{L_i}} =1 - \sqrt{\abs{\mu^{L_i}}^2} \leq 1 - \sqrt{1-4M\log ^{-(\gamma-1)}k} \leq 4 M  \log ^{-(\gamma-1)}k\leq \frac{1}{6}
$$
for $k$ large enough. Combining this with \eqref{eq:bound_arg}, we obtain
$
|\mu^j-1| \le 1 - |\mu^j| + |\arg \mu^j| \leq \frac{1}{2}
$
for $k$ large enough, which together with \eqref{eq:first_eq_lemma_linmu} yields the claim \eqref{eq:linear_small_angles}.

Let us turn to the second inequality and introduce $\nu_i = (1-\mu^{L_i})-L_i(1-\mu)$ for the approximation error of this linearization, for which we have the bound above. Then we have
\begin{align*}
  \left| \sum_i (1-\mu^{L_i})\right| &= \left|\sum_i L_i(1-\mu) + \sum_i\nu_i\right|\\
                                     &\ge \left|\sum_i L_i(1-\mu)\right| - \left|\sum_i\nu_i\right|\\
                                     &\ge |1-\mu| \sum_i L_i - \frac 1 2 \sum_i L_i |1-\mu|\\
                                     &=\frac 1 2|1-\mu| \left(\sum_i L_i\right) .
\end{align*}

Considering the third inequality we can write
\begin{align*}
  \sum_i\left|1-\mu^{L_i}\right|^2 &= \sum_i\left| L_i(1-\mu) + \nu_i\right|^2\\
                                   &\le \sum_i\left(|L_i(1-\mu)| + |\nu_i|\right)^2\\
                                   &\le \sum_i\left(|L_i(1-\mu)| + \frac 1 2 L_i|(1-\mu)|\right)^2\\
                                   &\le 3|1-\mu|^2 \left(\sum_i L_i^2\right).
\end{align*}
\end{proof}

%%%%%%%%%%%%%%%%%%%%%%%%%%%%%%%%%%%%%%%%%%%%%%%%%%%%%
%%%%%%%%%%%%%%%%%%%%%%%%%%%%%%%%%%%%%%%%%%%%%%%%%%%%%
%%%%%%%%%%%%%%%%%%%%%%%%%%%%%%%%%%%%%%%%%%%%%%%%%%%%%
\subsection{Approximate Eigenvector Decomposition}
%%%%%%%%%%%%%%%%%%%%%%%%%%%%%%%%%%%%%%%%%%%%%%%%%%%%%
%%%%%%%%%%%%%%%%%%%%%%%%%%%%%%%%%%%%%%%%%%%%%%%%%%%%%

The next lemma shows that any hypothetical solution $x$ to the modified eigenvalue equation
\eqref{eq:eig_main} with $\mu \in \forbid$ would be asymptotically
parallel to the vector $\1$. To lighten the notation, $\norm{.}$ will designate the 2-norm. Other norms will be explicitly specified. 

\providecommand{\xpar}{x^{\parallel}}
\providecommand{\xperp}{x^{\perp}}

\begin{lemma}\label{lem:parallel_1}
Let $x$ be a non-zero solution of \eqref{eq:eig_main} with $\mu\in \forbid$. 
Let $\xpar = (\frac{\1^Tx}{k})\1$ and $\xperp = x-\xpar$. Then under event $\BL$ there holds
\begin{equation}\label{eq:xperp<xpar}
  \norm{\xperp}^2  \leq \frac{\epsilon_k}{2\Delta_k-\Delta_k^2-\epsilon_k} \norm{\xpar}^2
  =:\delta_k \norm{\xpar}^2,
\end{equation}
where recall that $\epsilon_k=4M\log^{-(\gamma-1)} k$ is defined in \eqref{eq:abssq_bound} and $\Delta_k=c\log^{-\alpha} k$ is the lower bound for the spectral gap of $A$ (see Definition \ref{def:class_A}). Moreover, if $\gamma> \alpha + 1$, then  $\lim_{k\to \infty} \delta_k=0$
\end{lemma}

Observe that the denominator of the right hand part of \eqref{eq:xperp<xpar} is bounded from below and converges to a constant when $k$ grows. In the meantime the numerator converges to 0, so that $\delta_k$ converges to 0 a.a.s. (due to the event $\BL$) at a rate asymptotically comparable to that of $\epsilon_k$.
\begin{proof}
Taking the square norm on both sides of \eqref{eq:eig_main} leads to
\begin{equation}\label{eq:norm_equality}
\norm{\Dm x}^2 = \norm{CAx}^2 = \norm{Ax}^2 
\end{equation}
where we have use the fact that permuting entries with $C$ leaves the norm invariant. 
We begin by analyzing the right hand side. Since $A$ is doubly stochastic, observe that $A\xpar = \xpar$ is orthogonal to $A\xperp$. Hence there holds
\begin{align}
\norm{Ax}^2 & = \norm{A\xpar+A\xperp}^2 \nonumber\\
&= \norm{\xpar+A\xperp}^2\nonumber \\
&= \norm{\xpar}^2 + \norm{A\xperp}^2 \nonumber\\
&\leq \norm{\xpar}^2 + (1-\Delta_k)^2 \norm{\xperp}.
\label{eq:bound_nAx}
\end{align}
We now move to the left-hand part of \eqref{eq:norm_equality}. Using Lemma \ref{lem:muL_abs} and relying on $\BL$, we obtain
\begin{align}
\norm{\Dm x}^2&= \sum_i \abs{\mu^{L_i}}^2 |x_i^2|\nonumber\\
&\geq\sum_i  (1-\epsilon_k) |x_i^2|\nonumber \\
&= (1-\epsilon_k) \norm{x}^2\nonumber\\
&= (1-\epsilon_k)  \prt{\norm{\xpar}^2+\norm{\xperp}^2}. \label{eq:bound_nDx}
\end{align}
Reintroducing \eqref{eq:bound_nDx} and \eqref{eq:bound_nAx} in \eqref{eq:norm_equality} yields
$$
 (1-\epsilon_k)  \prt{\norm{\xpar}^2+\norm{\xperp}^2}  \leq  \norm{\xpar}^2 + (1-\Delta_k)^2 \norm{\xperp},
$$
and thus
$$
(1-\epsilon_k  -  (1-\Delta_k)^2  ) \norm{\xperp}^2  \leq  (1-(1-\epsilon_k)) \norm{\xpar}^2 ,
$$
which implies \eqref{eq:xperp<xpar}. The last part of the result follows from the definitions of $\Delta_k,\delta_k$ and $\epsilon_k$.
\end{proof}

As a consequence, under $\BL$ Lemma \ref{lem:parallel_1} implies that $\xpar\neq 0$, still for $\mu\in\forbid$.
Since solutions $x$ to \eqref{eq:eig_main} are defined up to a multiplicative constant, we can chose without loss of generality $\xpar = \1$ so that $\norm{\xpar}^2=k$. Equation \eqref{eq:xperp<xpar} becomes then
\begin{equation}\label{eq:xperp<dn}
\norm{\xperp}^2 \leq  k  \delta_k.
\end{equation}

%%%%%%%%%%%%%%%%%%%%%%%%%%%%%%%%%%%%%%%%%%%%%%%%%%%%%
%%%%%%%%%%%%%%%%%%%%%%%%%%%%%%%%%%%%%%%%%%%%%%%%%%%%%
%%%%%%%%%%%%%%%%%%%%%%%%%%%%%%%%%%%%%%%%%%%%%%%%%%%%%
\section{Small Angles}\label{sec:small_angles}
%%%%%%%%%%%%%%%%%%%%%%%%%%%%%%%%%%%%%%%%%%%%%%%%%%%%%
%%%%%%%%%%%%%%%%%%%%%%%%%%%%%%%%%%%%%%%%%%%%%%%%%%%%%

The main purpose of this section is to rule out non-trivial eigenvalues with large absolute value and small arguments a.a.s. More precisely, by the end we will prove the following:

\begin{proposition}
\label{prp:smallangle_main} $ $ \\
\noindent (i) Assume the conditions of Theorem \ref{thm:indep} hold. Then $G_L(A)$ has no eigenvalue in $\smallangles$ a.a.s.\ as $L,k\to \infty$.\\
\noindent (ii) Assume the conditions of Theorem \ref{thm:main} hold. Then
$H_n(A)$ has no eigenvalue in $\smallangles$ a.a.s.\ as $n,k\to \infty$.
\end{proposition}

%The analogous claim for the conditions of Theorem \ref{thm:indep} will be a convenient byproduct along the way.

%%%%%%%%%%%%%%%%%%%%%%%%%%%%%%%%%%%%%%%%%%%%%%%%%%
%%%%%%%%%%%%%%%%%%%%%%%%%%%%%%%%%%%%%%%%%%%%%%%%%%
\subsection{A Deterministic Criterion}\label{sec:aob}
%%%%%%%%%%%%%%%%%%%%%%%%%%%%%%%%%%%%%%%%%%%%%%%%%%

We first show that solutions to \eqref{eq:eig_main} with $\mu \in \smallangles$ can only exist if a certain condition on the first and second empirical moments of $(L_i)$ is satisfied. Note that this condition is deterministic given the values of $(L_i)$.

\begin{proposition}\label{prop:small_angles->L_i}
 Assume the event $\BL$ holds for the collection $(L_i)$, then the existence of a solution $x,\mu$ to \eqref{eq:eig_main} with $\mu\in \smallangles$ implies 
 the following inequality:
    \begin{equation}
      \label{eq:smallangle L Lsq}
      \prt{\frac{1}{k}\sum_{i} L_i}^2 \leq 12 \delta_k \prt{\frac{1}{k}\sum_{i} L_i^2    }.
    \end{equation}
\end{proposition}

\begin{proof}
Suppose there exists a solution $x,\mu$ to \eqref{eq:eig_main}, with $\mu\in \smallangles$. Pre-multiplying \eqref{eq:eig_main} by $\1^T$ shows that there would hold 
\begin{equation}
\1^T \Dm x = \1^T C Ax = \1^T Ax = \1^Tx  ,
\end{equation}
where we use that both the circulant $C$ and the matrix $A$ are column-stochastic and thus
$$
\1^T \Dm (\xpar  + \xperp) = \1^T(\xpar  + \xperp).
$$
Remembering $\Dm_{ii}= \mu^{L_i}$ and our scaling assumption $\xpar = \1$, we  have then
$$
\sum_i \mu^{L_i} (1+\xperp_i) =  \sum_i (1+\xperp_i),
$$
which implies
\begin{equation}\label{eq:small_a=square}
 | \sum_i (1-\mu^{L_i})|^2   =|\sum_i (1-\mu^{L_i}) \xperp_i|^2 .
\end{equation}
Now applying Cauchy-Schwartz together with Lemma \ref{lem:parallel_1} 
\begin{align*}
\left|\sum_i (1-\mu^{L_i}) \xperp_i\right|^2 & \leq  \prt{\sum_{i} |1-\mu^{L_i}|^2   } \norm{\xperp}^2\\
& \leq \delta_k k \prt{\sum_{i} |1-\mu^{L_i}|^2   }.
\end{align*}
Reintroducing this last inequality in \eqref{eq:small_a=square} and dividing by $k^2$ leads to 
$$
\left|\frac{1}{k} \sum_i (1-\mu^{L_i})\right|^2 \leq \delta_k \prt{\frac{1}{k}\sum_{i} |1-\mu^{L_i}|^2   }.
$$
Relying on $S(M)$ and $\mu\in\smallangles$ we may apply Lemma \ref{lem:linear_small_angles}, bound the left hand side below using \eqref{eq:linear_sum}, the right hand side above using \eqref{eq:linear_squares}, thus we get

$$
 \frac 1 4 \left|\frac{1}{k} \sum_i L_i (1-\mu) \right|^2 \leq 3 \delta_k \prt{\frac{1}{k}\sum_{i} |1-\mu|^2 L_i^2}.
$$
By construction $L_i$ are reals, which allows to rearrange as
$$
 |1-\mu|^2 \prt{\frac{1}{k} \sum_i L_i}^2 \leq 12 \delta_k |1-\mu|^2  \prt{\frac{1}{k}\sum_{i} L_i^2    },
$$
which implies the desired result. 
\end{proof}

%%%%%%%%%%%%%%%%%%%%%%%%%%%%%%%%%%%%%%%%%%%%%%%%%%%%%%%%%%%%
%%%%%%%%%%%%%%%%%%%%%%%%%%%%%%%%%%%%%%%%%%%%%%%%%%%%%%%%%%%%
\subsection{Impossibility for the \Arcmod Model}
%%%%%%%%%%%%%%%%%%%%%%%%%%%%%%%%%%%%%%%%%%%%%%%%%%%%%%%%%%%%

For the sake of clarity of exposition, we first compute separately two concentration inequalities on the empirical moments of $(L_i)$, which we will exploit to show that condition \eqref{eq:smallangle L Lsq} cannot be fulfilled when $k$ is large.

\begin{lemma}
\label{lm:L Lsq concentration}
  For the random arc lengths $(L_i)$ we have
\begin{align}
  \label{eq:sumL big}
  \PP \prt{\frac 1 k \sum_i L_i \leq \frac L 2} &\leq \frac 4 k,\\
  \label{eq:sumLsq small}
  \PP \prt{\frac 1 k \sum_i L_i^2 \geq 4L^2} &\leq \frac 5 k.
\end{align}
\end{lemma}
\begin{proof}
  Both inequalities follow from simple applications of Chebyshev's inequality based on the moments of the Geometric distribution.
  For the first inequality, consider
  $$
  \PP \prt{\frac 1 k \sum_i L_i \leq \frac L 2} \le \PP \prt{\left| \frac 1 k \sum_i L_i - L \right| \geq \frac L 2} \le \frac{\frac{1}{k} D^2(L_1)}{L^2/4} < \frac{4 L^2}{kL^2} = \frac 4 k.
  $$
  For the second one, we have to work with the i.i.d.\ series of random variables $L_1^2, L_2^2, \ldots,L_k^2$, for which $D^2(L_1^2) = E(L_1^4) - E(L_1^2)^2 < 20L^4$. %\comjh{what is $D$?}
  $$
  \PP \prt{\frac 1 k \sum_i L_i^2 \geq 4L^2} = \PP \prt{\left| \frac 1 k \sum_i L_i^2 - 2L^2 \right| \geq 2L^2} \le \frac{\frac 1 k D^2(L_1^2)}{4L^4} < \frac{20L^4}{4kL^4} = \frac 5 k.
  $$
\end{proof}

This lemma already allows deriving the first claim of Proposition \ref{prp:smallangle_main}.

\begin{proof}[Proof of Proposition \ref{prp:smallangle_main}(i)]
    Assume that neither of the two events of \eqref{eq:sumL big} or \eqref{eq:sumLsq small} appearing in Lemma \ref{lm:L Lsq concentration} hold, which is true with probability at least $1-9/k = 1-o(1)$. Then Proposition \ref{prop:small_angles->L_i} implies
\begin{equation*}
\frac{L^2}{4} \leq \prt{\frac{1}{k}\sum_{i} L_i}^2 \leq  12 \delta_k \prt{\frac{1}{k}\sum_{i} L_i^2} \leq 48 \delta_k L^2.
\end{equation*}
However, we know by Lemma \ref{lem:parallel_1} that $\delta_k\rightarrow 0$ as $k\rightarrow \infty$, therefore once $\delta_k < 1/192$ we get a contradiction showing that w.h.p.\ no $\mu\in \smallangles$ can be a solution to \eqref{eq:eig_main} for the \arcmod model.
\end{proof}

%%%%%%%%%%%%%%%%%%%%%%%%%%%%%%%%%%%%%%%%%%%%%%%%%%%%%%%%%%%%
%%%%%%%%%%%%%%%%%%%%%%%%%%%%%%%%%%%%%%%%%%%%%%%%%%%%%%%%%%%%
\subsection{Impossibility for the \Cyclemod Model}
%%%%%%%%%%%%%%%%%%%%%%%%%%%%%%%%%%%%%%%%%%%%%%%%%%%%%%%%%%%%

We now turn to the \cyclemod model. Our approach will be to consider it with respect to the
\arcmod model, conditioned to the total number of nodes being exactly $n$, which is justified by the following lemma,  similar to Proposition 3 in \cite{gerencser2019improved}.

\begin{lemma}\label{lem:modelconditioning}
The probability distribution for the eigenvalues of $H_n(A)$ is exactly the probability distribution for the eigenvalues of $G_L(A)$ conditioned to $\sum_i L_i=n$.
\end{lemma}
The idea of the proof is that the two (conditional) distributions are uniform and lead to transition matrices that have the same structure, but there is an additional technical element related to the extra degree of freedom of rotation in $H_n(A)$.
\begin{proof}

% announce of thing being entirely determined
Observe that for a given matrix $A$, the \arcmod matrix $G_L(A)$ is entirely determined by the realization of the $(L_i)$, wile the \cyclemod matrix $H_n(A)$ is entirely determined by the realization of the edges to be removed from the cycles and their label, see Definitions \ref{def:arcmod} and \ref{def:cyclemod}. 
% explanation we will build a bijection
We are going to build a bijection between the realizations of these edges to be removed and realizations of $(L_i)$ summing to $n$ concatenated with one number between $1$ and $n$, playing no role the transition matrix construction for \arcmod but that will be useful for our argument:
% build the f... bijection and observe it is indeed a bijection
For given $(L_i)$ summing to $n$ and the additional $j=1,\dots,n$ meant to represent the position of the first removed edge, we introduce the mapping $T$ on $((L_i),j)$ %to the \cyclemod model $T((L_i),j)$ 
by setting $[b_1,e_1]=[j,j+L_1-1]$, then recursively $[b_i,e_i]=[e_{i-1}+1,e_{i-1}+L_i]$, identifying $n+p$ to $p$ if the indices exceed $n$. This mapping is clearly a bijection, noting that we distinguish the vertices of $A$. Moreover, it preserves the eigenvalues of the matrix (and indeed, the structure of the graph). 

% isomorphism
Observe now that the \arcmod transition matrix obtain from the $(L_i)$ is isomorphic to the \cyclemod transition matrix obtained from $T((L_i),j)$ (and vice-versa), and have in particular the same eigenvalues.

%uniformity
To conclude, we also ensure the probability distributions are being matched.
Observe that the probability of a realization of $G_L(A)$ when $\sum_i L_i=n$ is $L^{-k}(1-\frac 1L)^{n-k}$, consequently it is uniform once conditioned on the event $\{\sum_i L_i = n\}$.
Applying the induced measure transformation $T^*$ we get the uniform distribution. Therefore, this matches the distribution of $H_n(A)$, which is uniform by construction. Consequently, combined with the isomorphism claim above, this shows the distribution of eigenvalues are equal as well.
\end{proof}

We will exploit this equivalence to extend claim (i) of Proposition \ref{prp:smallangle_main} to the \cyclemod model by showing that the high probability events discussed for the \arcmod model keep a high probability when conditioned to the sum of the arc lengths being $n$, for $L=n/k$. For this purpose, we need to establish two claims. 

\begin{lemma}\label{lem:exceed_avg_04}
For $k, L= \frac nk$ large enough, for the \arcmod model
we have
$$
P\left(\sum_i L_i \geq n\right) \geq 0.4.
$$ 
\end{lemma}
\begin{proof}
  We exploit the Berry-Esseen theorem ensuring that the sum of random variables in question are sufficiently close to a Gaussian, therefore being above their expectation has probability around $1/2$, as needed. Let us center the random variables, defining $L_i' = L_i - L$. 
  Knowing that $L_i$ are Geometric random variables, we have
  $$
  \E L_i' = 0, \quad \E L_i'^2 = L^2\left(1-\frac 1 L\right).
  $$
We rewrite the complement of the probability to bound as $\PP(\sum_i L_i < kL = n) = \PP(\sum_i (L_i-L)<0)$. To get a handle on it, we apply the Berry-Esseen theorem at position $0$ to get
  $$
  \left| \PP\left(\frac{\sum_i L_i'}{\sqrt{k}DL_i'} < 0 \right) - \Phi(0) \right|
  \le \frac{C \E|L_i'|^3}{\sqrt{k} D^3( L_i')}
  $$
  for some global positive constant $C$. We now estimate different quantities involved in this bound. In general, for some $X\sim Geo(p)$ it is known that
  $
  \E X^3 = \frac{p^2-6p+6}{p^3}
 $. In our case, since $L_i\sim Geo(1/L)$, we have thus for growing $L$
 $$
  \E|L_i'|^3 = \E|L_i-L|^3 \le \E L_i^3 + 3L \E L_i^2 + 3L^2 \E L_i + L^3 = (16+o(1))L^3.
  $$
At the same time, $D^2(L'^2) = \E L_i'^2$, so  together with $\Phi(0)=1/2$ we obtain
 $$
  \left|\PP\left(\sum_i (L_i-L) < 0\right) - 1/2\right|
  \le \frac{C (16+o(1))L^3}{\sqrt{k} (1+o(1)L^2)^{3/2}} = \frac{16C(1+o(1))}{\sqrt{k}}.
  $$
  which gets arbitrarily small when $k$ grows, and in particular smaller than 0.1 if $k$ is sufficiently large, implying the desired result.
\end{proof}

\begin{lemma}\label{lem:increase_cond_proba}
  For any $K,m\in\Z^+$ there holds
  $$
  \PP\left(\sum_i L_i^2\geq K \mid \sum_i L_i  = m+1\right) \geq \PP\left(\sum_i L_i^2\geq K \mid \sum_i L_i  = m\right).
  $$
\end{lemma}
\begin{proof}
Recall that  all realizations $(L_1,\dots,L_k)$ for which $\sum_{i} L_i= m$ (resp.\ $m+1$) have the same probability, in other words, the conditional distributions are uniform (see the proof of Lemma \ref{lem:modelconditioning}).

We are going to jointly construct coupled random variables $J,J^+$ on $\{(j_i)\in\Z_+^k~|~\sum j_i = m\}$ and $\{(j^+_i)\in\Z_+^k~|~\sum j^+_i = m+1\}$ whose marginals are also uniform on their range: 
For a pair $((j_i),(j^+_i))$ we assign a positive weight if they agree on all but one coordinate, where we have $j_t+1 = j^+_t$. In this case the weight given is $j_t$. We get a distribution on the pairs of $k$-tuples simply by normalizing these weights, and sample $(J,J^+)$ from this distribution.

We now confirm that the marginals of this distribution are indeed uniform on their range. Concerning $J$, a realization $(j_i)$ will appear in $k$ pairs with positive weights, corresponding to the $k$ possibilities of increasing a coordinate,
leading to a total weight of
$$\sum_{i} j_i = m.$$
This is indeed the same on all realizations, consequently the marginal for $J$ is uniform. 
Concerning $J^+$, each $j^+$ appears in up to $k$ pairs, corresponding to all the ways of decreasing a coordinate larger than 1. This has total weight 
$$\sum_{i} (j_i^+-1) = m+1-k,$$
where the coordinates $j^+_i=1$ do not need to be explicitly excluded as their contribution to the sum is zero. In the end, we once again get a constant for all realizations of $J^+$, confirming its uniform distribution.

By construction we have $J^+\geq J$ elementwise for any realization. Together with the fact that that the conditional distribution of $(L_1,\dots,L_k)$ is exactly the distribution of $J$ (resp.\ $J^+$) when conditioned on the sum being $m$ (resp.\ $m+1$), this implies
$$ 
\PP(\sum_i L_i^2\geq K \mid \sum_i L_i  = m+1) = \PP(\sum_i J_i^{+2} \ge K) \ge \PP(\sum_i J_i^2 \ge K) = \PP(\sum_i L_i^2\geq K \mid \sum_i L_i  = m). $$
\end{proof}

Now that we have built the link between the \arcmod and \cyclemod model and have collected some technical bounds related to the conditional probabilities involved, we are ready to prove the eigenvalue impossibility in $\smallangles$ for the \cyclemod model.
\providecommand{\elas}{\mathcal{E}}% even large average square
\begin{proof}[Proof of Proposition \ref{prp:smallangle_main}(ii)]
Let $\elas$ denote the event $\sum_i L_i^2 > 4kL^2$. It follows from Lemma \ref{lm:L Lsq concentration} that $\PP(\elas)\leq \frac{5}{k}$. Moreover, 
the assumption of Theorem \ref{thm:main} imply that $L=n/k \to \infty$ when $n$ grows, so that Lemma \ref{lem:exceed_avg_04} applies. Together with Lemma \ref{lem:increase_cond_proba}, it shows that 
$\PP(\sum_i L_i^2 \mid \sum_i L_i = m)$ increases with $m$, and also that $P(\sum_i L_i \geq n) \geq 0.4$. Combining these we get
\begin{align*}
\frac{5}{k} &\geq \PP(\elas ) \\&= \PP(\elas \mid \sum_i L_i < n)\PP(\sum_i L_i <n) + \PP(\elas \mid \sum_i L_i = n)\PP(\sum_i L_i =n) +\PP(\elas \mid \sum_i L_i > n)\PP(\sum_i L_i >n)\\
&\geq   \PP(\elas \mid \sum_i L_i < n)\PP(\sum_i L_i <n) + \PP(\elas \mid \sum_i L_i = n)\bigg(\PP(\sum_i L_i =n)+\PP(\sum_i L_i >n)\bigg)\\
& \geq  0 + \PP(\elas \mid \sum_i L_i = n)\cdot 0.4.
\end{align*}
Hence 
\begin{equation}
    \label{eq:Lsq_cond}
    \PP\left(\sum_i L_i^2 > 4kL^2 \mid \sum_i L_i = n\right) \le \frac{25}{2k}.
\end{equation}

Let us now consider on the possible existence of an eigenvalue $\mu\in \smallangles$. Proposition \ref{prop:small_angles->L_i} shows this is possible only if condition \eqref{eq:smallangle L Lsq} holds, or if $\BL$ does not hold. Focusing first on condition \eqref{eq:smallangle L Lsq}, we see that when $\sum_i L_i= n$, it becomes  $\sum_i L_i^2 \geq \frac{k}{12\delta_k}L^2$. 
Therefore the probability that it holds can be rewritten as
$$
\PP\prt{\frac{k}{12 \delta_k}L^2 \le \sum_{i} L_i^2 \mid \sum_i L_i = n}.
$$
Knowing that $\delta_k\rightarrow 0$ as $k\to\infty$ (see Lemma \ref{lem:parallel_1}), the coefficient on the left hand side in front of $kL^2$ eventually increases above $4$, so this probability is $o(1)$ by \eqref{eq:Lsq_cond}, excluding the possibility of a solution w.h.p.
Coming back to $\BL$, we have seen in Lemma \ref{lem:bound_Li} that the probability that it does not hold in the \arcmod model is at most $2k^{1-M}$. At the same time, Lemma 5 in \cite{gerencser2019improved} established that $\PP(\sum_i L_i=n)\geq 1/n$. Relying on the polynomial relation between $n,k$ required - inherited from Theorem \ref{thm:main} - i.e. $\rho_l\le\log n/\log k\le\rho_u$, if $M>1+\rho_u$, then $\PP(S(M)) = o(1/n)$, hence $S(M)$ must have asymptotically vanishing probability even when conditioned on $\sum_i L_i=n$. Therefore the probability of $S(M)$ not to hold and that of condition \eqref{eq:Lsq_cond} holding are both vanishingly small for the \cyclemod model, implying by Proposition \ref{prop:small_angles->L_i} that a.a.s. there is no eigenvalue in $\smallangles$.
\end{proof}

%%%%%%%%%%%%%%%%%%%%%%%%%%%%%%%%%%%%%%%%%%%%%%%%%%%%%
%%%%%%%%%%%%%%%%%%%%%%%%%%%%%%%%%%%%%%%%%%%%%%%%%%%%%
%%%%%%%%%%%%%%%%%%%%%%%%%%%%%%%%%%%%%%%%%%%%%%%%%%%%%
\section{Large Angles}\label{sec:large_angles}
%%%%%%%%%%%%%%%%%%%%%%%%%%%%%%%%%%%%%%%%%%%%%%%%%%%%%
%%%%%%%%%%%%%%%%%%%%%%%%%%%%%%%%%%%%%%%%%%%%%%%%%%%%%

To complement to the previous section, we now show that 
there are no eigenvalues with large absolute values and large arguments a.a.s. More precisely, we aim for the following:

\begin{proposition}
\label{prp:largeangle_main} $ $ \\
\noindent (i) Assume the conditions of Theorem \ref{thm:indep} hold. Then $G_L(A)$ has no  eigenvalue in $\largeangles$ a.a.s.\ as $L,k\to \infty$.\\
\noindent (ii) Assume the conditions of Theorem \ref{thm:main} hold. Then $H_n(A)$ has no eigenvalue in $\largeangles$ a.a.s.\ as $n,k\to \infty$.
\end{proposition}

%%%%%%%%%%%%%%%%%%%%%%%%%%%%%%%%%%%%%%%%%%%%%%%%%%%%%
%%%%%%%%%%%%%%%%%%%%%%%%%%%%%%%%%%%%%%%%%%%%%%%%%%%%%
\subsection{Polynomial Analysis}
%%%%%%%%%%%%%%%%%%%%%%%%%%%%%%%%%%%%%%%%%%%%%%%%%%%%%

The following polynomial, which is a scaled version of the trace of $\Dm$ appearing in our modified eigenvalue equation \eqref{eq:eig_main}, will play an important role in our analysis.
$$
P(\mu) = \frac{1}{k}\sum_{i=1}^k \mu^{L_i}
$$

Analogously to Section \ref{sec:small_angles}, we will show that some deterministic condition must be satisfied if there exists an eigenvalue of magnitude close to 1, before showing that this condition is a.a.s. not satisfied for any  $\mu \in \largeangles$. 
\begin{lemma}
\label{lem:Pmu-near-1}
Under the event $\BL$, if there exists a nontrivial solution $(x,\mu)$ to the modified eigenvalue problem \eqref{eq:eig_main} with 
$\mu \in \forbid$, then
$$
\abs{1- P(\mu)} \leq \sqrt{\delta_k},
$$
where $\delta_k$ is defined in Lemma 
\ref{lem:parallel_1}.
\end{lemma}
\begin{proof}
As before, we assume without loss of generality that $x=\xpar+\xperp$ with $\xpar = \1$, so that $\1^\top x=k$, which is possible because $\xpar \neq 0$, see Lemma \ref{lem:parallel_1}.
Left-multiplying \eqref{eq:eig_main} by $\1^\top$ leads then to
\begin{equation}\label{eq:1Tdmx=sqrt(k) a}
\1^\top \Dm x = \1^\top C Ax = \1^\top x = k,
\end{equation}
because $A$ is column stochastic and $C$ is a permutation matrix.
The left-hand side can be rewritten as
$$
\1^\top  \Dm (\xpar + \xperp) = \sum_i \mu^{L_i}+ \sum_i \mu^{L_i}\xperp_i = k P(\mu) + \sum_i \mu^{L_i}\xperp_i,
$$
so that \eqref{eq:1Tdmx=sqrt(k) a} becomes
\begin{equation}
  \label{eq:kP=sumxperp a}
  k(1-P(\mu)) = \sum_i \mu^{L_i}\xperp_i.
\end{equation}
Since $\mu\in \forbid$ implies $\abs{\mu}\leq 1$, the vector containing all the $\mu^{L_i}$ has a 2-norm at most $\sqrt{k}$. Therefore, Cauchy-Schwartz implies 
$$
\abs{\sum_i \mu^{L_i}\xperp_i} \leq \sqrt{k} \norm{\xperp}_2.
$$
Moreover, under $\BL$, and using $\xpar = \1$, Lemma \ref{lem:parallel_1} implies $\norm{\xperp}_2 \leq \sqrt{k\delta_k}$.
Combining this with \eqref{eq:kP=sumxperp a} leads then to the desired inequality:
$$
k\abs{1- P(\mu)} \leq \abs{\sum_i \mu^{L_i}\xperp_i} \leq k \sqrt{\delta_n}.
$$
\end{proof}

To obtain a contradiction, we now show that for $\mu\in \largeangles$, $P(\mu)$ is a.a.s.\ sufficiently far from 1, first for the \arcmod model. Intuitively, this is because it can take the value 1 only if all the $\mu^{L_i}$ are aligned, and their angles will almost surely be sufficiently different from each other. This is formalized in the following result, which is parallel to a development in \cite{gerencser2019improved} involving $P(\mu^{-1})$ due to a different parametrization.

\begin{lemma}\label{lem:paper_polynomial}
Consider the \arcmod model, and fix any constants $\eta >3, \theta>1$.
There is a probability at least $1-e^{-\Omega(k/\log^\theta k)}$ that for any $\mu \in \largeangles$, 
$$
\Re(P(\mu))\leq 1 - \frac{1}{\log^\eta k} .
$$
\end{lemma}

The proof uses the following proposition .
\begin{proposition}[Proposition 2, \cite{gerencser2019improved}]\label{prop:P_partial_from_paper}
    \label{prp:cosloss}
    Consider the \arcmod model together with the assumptions of Theorem \ref{thm:indep}.
    For any constants $\theta,\beta > 1$, let
    $m_k = \lceil \log^\theta k \rceil$ and $\sigma_k = \log^{-\beta} k.$
    Then for $k,L$ large enough we have 
        $$\PP\left(\sup_{\mu\in \largeangles}\sum_{i=1}^{m_k} \cos^+(L_i  \arg(\mu)) < m_k - \sigma_k^2\right)
    \ge \frac{1}{3},$$
    where $\cos^+(y) = \max(\cos(y),0).$ 
     \end{proposition}

\begin{proof}[Proof of Lemma \ref{lem:paper_polynomial} ]

Since $\eta>3$, wet can chose $\theta, \beta\geq 1$ such that
  $\theta+2\beta < \eta$, possibly decreasing the originally given $\theta$ leading to a stronger probability bound.
  %, which is clearly possible.
Taking into account $|\mu|\le 1$, we have
$$
\Re\left( \sum_{i=1}^{m_k} (\mu^{L_i} -1)\right) 
\le \sum_{i=1}^{m_k} |\mu|^{L_i}\cos^+(L_i  \arg(\mu)) -m_k
\le \sum_{i=1}^{m_k} \cos^+(L_i  \arg(\mu))-m_k,
$$
which is always non-positive, and by Proposition \ref{prop:P_partial_from_paper} has a probability at least $1/3$ of being uniformly bounded from above by $-\sigma_k^2$ for all $\mu \in \largeangles$.
Observe now that $kP(\mu)-k= \sum_{i=1}^{k} (\mu^{L_i} -1)$.
We cut this sum into blocks of size
$m_k$ each, except possibly the last one, where we bound the real part above by 0.
For each block, let $\xi_j$ be the indicator that the event described above hold, so that the part of $\Re\left( \sum_{i=1}^{m_k} (\mu^{L_i} -1)\right) $ corresponding to each block $j$ is bounded by $-\sigma_k^2 \xi_j$, and therefore 
$$
k\Re(P(\mu)) - k \le -\sigma_k^2\sum_{j=1}^{\floor{k/m_k}} \xi_j.
$$
Moreover, the independence of the $(L_i)$ $\xi_j$ in the \arcmod model implies the independence of the $\xi_j$, which are thus i.i.d.\ indicators with probability at least $1/3$.

Standard concentration results imply then $\sum_{j=1}^{\floor{k/m_k}}
\xi_j \ge \frac{k}{4m_k}$ with a probability at least $1-e^{-\Omega(k/m_k)} = 1-e^{-\Omega(k/\log^\theta k)}$.
Substituting this high probability bound and rearranging we get
$$
\Re(P(\mu)) \le 1 - \frac{\sigma_k^2}{4m_k} \qquad \forall \mu\in\largeangles,
$$
and the claim then follows directly from the definition of  $\sigma_k,m_k$.
\end{proof}

%%%%%%%%%%%%%%%%%%%%%%%%%%%%%%%%%%%%%%%%%%%%%%%%%%%%%
%%%%%%%%%%%%%%%%%%%%%%%%%%%%%%%%%%%%%%%%%%%%%%%%%%%%%
\subsection{Impossibility of Large Eigenvalue with Large Angles}
%%%%%%%%%%%%%%%%%%%%%%%%%%%%%%%%%%%%%%%%%%%%%%%%%%%%%

We are now ready to prove Proposition \ref{prp:largeangle_main} by exploiting the contradictions 
between the high probability bounds on $P(\mu)$ and by controlling the exception probabilities.

\begin{proof}[Proof of Proposition \ref{prp:largeangle_main}]
Let us consider claim (i) on the \arcmod model first.
By Lemma \ref{lem:paper_polynomial}, the following there is a probability at least $1-e^{-\Omega(k/\log^\theta k)}$ that the following implication holds: if there is an eigenvalue $\mu\in \largeangles$ then $\Re(P(\mu))\leq 1 - \frac{1}{\log^\eta k}$.
On the other hand, with probability at least $1- k ^{1-M}$ event $\BL$ holds, in this case Lemma \ref{lem:Pmu-near-1} implies $\abs{1- P(\mu)} \leq \sqrt{\delta_k}$ for any $\mu\in \largeangles$, and thus
\begin{equation}
    \label{eq:logcompare}
\frac{1}{\log^\eta k} \leq 1-\Re(P(\mu)) \leq \sqrt{\delta_k}.
\end{equation}
To compare the magnitudes of the two sides, recall the definition of  $\delta_k = \frac{\epsilon_k}{2\Delta_k-\Delta_k^2-\epsilon_k},$ 
with $\epsilon_k = 4M\log^{-(\gamma-1)} k$, see Lemma \ref{lem:parallel_1}. Thus $\delta_k=O(\log^{-(\gamma-1-\alpha)}k)$, so the requirement \eqref{eq:logcompare} takes the form
$$
\log^{-\eta} k \le \sqrt{\delta_k} = O(\log^{-\frac 12 (\gamma-1-\alpha}),
$$
and we get a contradiction 
as soon as $(\gamma-1-\alpha)/2 > \eta$, or equivalently, $\gamma > 2\eta+\alpha+1$.
Lemma \ref{lem:paper_polynomial} required $\eta>3$, thus
if we fix any $\gamma >7+\alpha$ we can find a good $\eta >3$ satisfying the required inequality between the two.
Altogether, with such parameters we have a contradiction excluding the existence of eigenvalues in $\largeangles$ with probability at least $1-e^{-\Omega(k/\log^\theta k)}-k ^{1-M}$.

We now move to claim (ii) on the \cyclemod model. We have seen in Lemma \ref{lem:modelconditioning} that the distribution of eigenvalues for $H_n(A)$ is exactly the same as that for $G_L(A)$ conditioned to the event $\sum_i L_i=n$, and Lemma 5 in \cite{gerencser2019improved} establishes that for $L=n/k$, the probability for this event is at least in the \arcmod model is at least $1/n$.
Hence, even conditioned to that event, the probability of $G_L(A)$ - and thus $H_n(A)$ - to have an eigenvalue in $\largeangles$ is asymptotically vanishingly small, provided that bound on
the (unconditionnal) probability of in the \arcmod context to have an eigenvalue in $\largeangles$, $e^{-\Omega(k/\log^\theta k)}+k ^{1-M}$, is $o(1/n)$. Thanks to the assumptions of Theorem \ref{thm:main} requiring $\rho_u \log k > \log n $, by choosing $M> \rho_u+1$ the second term is $o(1/n)$, and the same holds for the first term thanks again to $\rho_u \log k > \log n $.
\end{proof}

In view of the definition of $\smallangles$ and $\largeangles$, 
Theorems \ref{thm:indep} and \ref{thm:main} then follow from the combination of Propositions \ref{prp:smallangle_main} and \ref{prp:largeangle_main}.

%%%%%%%%%%%%%%%%%%%%%%%%%%%%%%%%%%%%%%%%%%%%%%%%%%%%%
%%%%%%%%%%%%%%%%%%%%%%%%%%%%%%%%%%%%%%%%%%%%%%%%%%%%%
%%%%%%%%%%%%%%%%%%%%%%%%%%%%%%%%%%%%%%%%%%%%%%%%%%%%%
\section{Numerical Analysis }\label{sec:numerics}
%%%%%%%%%%%%%%%%%%%%%%%%%%%%%%%%%%%%%%%%%%%%%%%%%%%%%
%%%%%%%%%%%%%%%%%%%%%%%%%%%%%%%%%%%%%%%%%%%%%%%%%%%%%

While Theorem \ref{thm:main} provides a lower bound on the spectral gap, no matching upper bound is proven, to get a hint on the true behavior we have carried out numerical approximations.
We had $\log n$ range from $4.62$ to $8.0$, corresponding to $n\in [101, 2980]$. For the interconnection points, we have chosen $\log k \approx \frac 12 \log n$, i.e. $k\approx \sqrt n$ and for simplicity we restrict $k$ to be even, thus allowing interconnection graphs with odd average degrees.
To complete the model parametrization we need to specify the interconnection structures, for which multiple choices have been tested:
\begin{enumerate}[label=(\alph*)]
    \item Complete graph;
    \item Random regular graph with degree $k/2$;
    \item Random regular graph with degree 4;
    \item Cycle plus random perfect matching (on the $k$ vertices), the Bollob\'as-Chung model \cite{bollobas1988diameter}.
\end{enumerate}
In this collection $(a)$ corresponds to the simplified situation studied in \cite{gerencser2019improved} and is used for reference. The variant $(b)$ can be viewed as an intermediary choice with non-trivial connection structure, but still with $\lambda \approx 1$. Then $(c),(d)$ are standard choices for sparse expanders with $\lambda = O(1)$ thus still in $\cA_{c,0}$ for some $c>0$.
\begin{figure}[h!]
    \centering
    \resizebox{0.5\textwidth}{!}{\input{loglog_lambda_alpha05_all}}
    \caption{log-log plot of spectral gap for various interconnections, $(a),(b),(c),(d)$ choices from top to bottom.}
    \label{fig:loglog_lambda}
\end{figure}
The aggregated results of 500 realizations of each system size $n$ is presented in Figure \ref{fig:loglog_lambda}. A stripe is drawn between the $1^{st}$ and $3^{rd}$ quartile to exclude outliers for each of them.
A linear trend can be observed with comparable slope for each interconnection structure, mostly differing by a shift. This indicates that the lower bound of Theorem \ref{thm:main} appears to predict, up to polylogarithmic factors, the actual dependency on $n$. One can also observe an a priori surprising seesaw pattern, which is stronger for smaller $n$. It is actually a consequence of the discrete and slow evolution of $k$ in our construction. Indeed, the number of interconnection points $k$ increases slowly as compared to the growth of the total number of vertices $n$, and since it has to be an integer, it will inevitably step only at some values of $n$. Such a relative jump results in a sudden improvement of $\lambda$, and then as long as $k$ remains constant, it results in a slow deterioration of $\lambda$ with the growth $n$. These two effects build up the apparent seesaw aspect.

\begin{figure}[h!]
    \centering
    \resizebox{0.5\textwidth}{!}{\input{lambdaL_alpha05_all}}
    \caption{Plot of compensated spectral gap for various interconnections, $(a),(b),(c),(d)$ choices from top to bottom.}
    \label{fig:lambdaL}
\end{figure}

To get a closer view on the linear trend we plot $L\lambda$ in Figure \ref{fig:lambdaL}, using similar stripes as before.
Comparing with the $1/(L \log^\gamma k)$ bound obtained analytically would suggest a polynomial decay in the residual term (in terms of $\log n$). However, for $(c),(d)$ we obtain a striking horizontal stripe suggesting an actual order of $1/L$. For the interconnections $(a),(b)$ we even get slightly upwards going stripes. The complementary analysis of case $(a)$ in \cite{bbsztdk2021} includes the upper bound of order $1/L$, so the visible upwards trend is likely due to the slower stabilization, and relatively worse small-scale performance. In any case the figure indicates the typical behavior do not include the additional polynomial term.
Of course, one should bear in mind that a gap could exist between the central quartiles represented in Figures \ref{fig:loglog_lambda} and \ref{fig:lambdaL} and the asymptotically almost sure bounds established in Theorem \ref{thm:main}.

%%%%%%%%%%%%%%%%%%%%%%%%%%%%%%%%%%%%%%%%%%%%%%%%%%%%%
%%%%%%%%%%%%%%%%%%%%%%%%%%%%%%%%%%%%%%%%%%%%%%%%%%%%%
%%%%%%%%%%%%%%%%%%%%%%%%%%%%%%%%%%%%%%%%%%%%%%%%%%%%%
\section{Conclusions}\label{sec:ccl}
%%%%%%%%%%%%%%%%%%%%%%%%%%%%%%%%%%%%%%%%%%%%%%%%%%%%%
%%%%%%%%%%%%%%%%%%%%%%%%%%%%%%%%%%%%%%%%%%%%%%%%%%%%%

We have shown that polynomially rare interconnection points with appropriate, possibly sparse, structure added is sufficient to reveal mixing on the order of the diameter (up to logarithmic factors), provided the transition probabilities are adapted in a strongly non-reversible way, generalizing the results in \cite{gerencser2019improved} which only applied to fully connected added interconnections.

There are several natural open questions following this work.
First, the statement of Theorem \ref{thm:main} is independent of the specific properties of the interconnection structure introduced provided its own spectral gap remains bounded away from zero. The impact of the precise structure of this interconnection, which can be observed in Figures \ref{fig:loglog_lambda} and \ref{fig:lambdaL}, remains unexplored. Note that an improvement of the polylogarithmic factors involved in the bound would also be welcome.

Second, our result show that the cycle exploited in a fully asymmetric manner by a deterministic walk, when combined with the addition of long-distance edges, leads to a much faster convergence than for the fully symmetric walk on the cycle with the same added edges. However, this choice prevents any local diffusion along the arcs between meeting the additional edges. Hence a natural question would be whether an even better spectral gap could be obtained for similar Markov chains with non-trivial, but still non-reversible transition probability structures for the bulk of the vertices.

A third natural direction is towards the
refined understanding of the mixing behavior, more specifically getting analogous mixing time bounds to the current spectral gap approximations and determining whether cutoff occurs.
We should note that while a deterministic walk could be handled, this required separation of outgoing and incoming edges at a connection point. Therefore it would be interesting to see the variant with a lazy chain along the cycle, but single connection points.

Finally, the long term goal is to carry further and abstract the results for more setups to be able to exploit the combination of minor perturbation of allowed connections and non-reversible modification of a reference Markov chain to reach substantially improved mixing.

Such progress could have a notable impact on multiple processes building upon Markov kernels, such as consensus, push-sum, or decentralized optimization, many algorithms which can be seen as Markov chains constantly perturbed by gradient-type information at each node.

\bibliography{references} % Entries are in the refs.bib file

\end{document}